\documentclass{amsart}

\usepackage[latin1]{inputenc}
\usepackage{amssymb}
\usepackage{amsmath}
\usepackage{enumerate}
\usepackage{enumitem}
\usepackage{epsfig}
\usepackage{subfigure}
\usepackage[all]{xy}
\usepackage{longtable}
\usepackage{color}
\usepackage{hyperref}

\newtheorem{theorem}{Theorem}[section]
\newtheorem{recalltheorem}{Theorem}
\newtheorem{lemma}[theorem]{Lemma}
\newtheorem{proposition}[theorem]{Proposition}
\newtheorem{corollary}[theorem]{Corollary}

\theoremstyle{definition}

\newtheorem{example}[theorem]{Example}

\theoremstyle{remark}
\newtheorem{rem}[theorem]{Remark}

\numberwithin{equation}{section}

\newcommand{\C}{\mathbb{C}}
\newcommand{\I}{\mathbb{I}}

\newcommand{\T}{\mathbb{T}}
\newcommand{\Z}{\mathbb{Z}}

\newcommand{\PP}{\mathbb{P}}

\newcommand{\cA}{\mathcal{A}}

\newcommand{\cO}{\mathcal{O}}
\newcommand{\cZ}{\mathcal{Z}}

\newcommand{\im}{\operatorname{im}}

\newcommand{\Span}{\operatorname{Span}}

\newcommand{\del}{\partial}
\newcommand{\delbar}{\bar{\del}}
\newcommand{\blank}{\underline{\quad}}

\begin{document}

\date{\today}

\title{Complex symplectic structures and the $\del\delbar$-lemma}

\author{Andrea Cattaneo}
\address{Dipartimento di Scienza e Alta Tecnologia\\
Universit\`a degli Studi dell'Insubria\\
Via Valleggio, 11, 22100\\
Como, Italy}
\email{andrea1.cattaneo@uninsubria.it}

\author{Adriano Tomassini}
\address{Dipartimento di Scienze Matematiche, Fisiche, ed Informatiche\\
Plesso Matematico e Informatico\\
Universit\`a di Parma\\
Parco Area delle Scienze, 53/A, 43124\\
Parma, Italy}
\email{adriano.tomassini@unipr.it}

\thanks{The first author is granted with a research fellowship by Istituto Nazionale di Alta Matematica INdAM, and is supported by the Project PRIN ``Variet\`a reali e complesse: geometria, topologia e analisi armonica'', by the Project FIRB ``Geometria Differenziale e Teoria Geometrica delle Funzioni'', and by GNSAGA of INdAM. The second author is supported by the Project PRIN ``Variet\`a reali e complesse: geometria, topologia e analisi armonica'' and by GNSAGA of INdAM}

\begin{abstract}
In this paper we study complex symplectic manifolds, i.e., compact complex manifolds $X$ which admit a holomorphic $(2, 0)$-form $\sigma$ which is $d$-closed and non-degenerate, and in particular the Beauville--Bogomolov--Fujiki quadric $Q_\sigma$ associated to them. We will show that if $X$ satisfies the $\del \delbar$-lemma, then $Q_\sigma$ is smooth if and only if $h^{2, 0}(X) = 1$ and is irreducible if and only if $h^{1, 1}(X) > 0$.
\end{abstract}

\maketitle

\section*{Introduction}

Let $M$ be a compact real manifold, endowed with both a complex structure $I$ and a compatible Riemannian metric $g$. Then the $2$-form associated to $(M, I, g)$, namely $\omega_g(\blank, \blank) = g(I\blank ,\blank)$, can be seen as a collection of non-degenerate $2$-forms on each real tangent space to $M$, varying smoothly with the point. If $\omega_g$ is $d$-closed, then the triple $(M, I, \omega_g)$ is a K\"ahler manifold.

Let now $X$ be a compact \emph{complex} manifold, and assume we have a collection of non-degenerate $2$-forms, one on each holomorphic tangent space, varying holomorphically with the point. This is then a $(2, 0)$-form on $X$, and under the further assumption that it is $d$-closed, it is called a \emph{complex symplectic structure} (cf.\ \cite[p.\ 763]{bea} and \cite{boalch}).

In the context of K\"ahler geometry, such kind of manifolds plays an important role: by the Bogomolov covering theorem, any compact K\"ahler manifold with vanishing first Chern class has a covering which splits as the product of Calabi--Yau manifolds, complex tori and irreducible holomorphic symplectic manifolds. Among these, the last two are in fact complex symplectic manifolds. In particular, an interesting and crucial tool for the study of the latter is the \emph{Beauville--Bogomolov--Fujiki quadratic form} introduced in \cite{bea} (see \eqref{eq: bbf} for the definition).

The aim of this paper is to study cohomological properties of compact complex (possibly non-K\"ahler) manifolds $X$ endowed with a complex symplectic structure $\sigma$. We weaken the K\"ahler assumption requiring that $X$ satisfies the $\del \delbar$-lemma (cf.\ \cite{dgms}). Observing that the Beauville--Bogomolov--Fujiki form $q_\sigma$ makes sense also in this setting, we will show that if $X$ satisfies the $\del \delbar$-lemma, then there is a deep link between the geometric properties of the quadric $Q_\sigma$ defined in $\PP(H^2_{dR}(X, \C))$ by the vanishing of $q_\sigma$ and the Dolbeault cohomology of $X$: the smoothness of $Q_\sigma$ depends on the Dolbeault spaces $H^{2, 0}_{\delbar}(X)$ and $H^{0, 2}_{\delbar}(X)$, while its irreducibility depends on $H^{1, 1}_{\delbar}(X)$.

The structure of the paper is as follows. In Section \ref{sect: preliminaries} we set up the notation and recall the basic facts we will use throughout the paper. Section \ref{sect: bbf form} is devoted to the proof of the two main theorems:

\begin{recalltheorem}[{cf.\ Thm.\ \ref{thm: irreducible}}]
Let $(X, \sigma)$ be a complex symplectic manifold, and assume that $X$ satisfies the $\del \delbar$-lemma. Let $Q_\sigma$ be the quadric in $\PP(H^2_{dR}(X, \C))$ defined by \eqref{eq: bbf}. Then the following are equivalent:
\begin{enumerate}
\item $h^{1, 1}(X) > 0$;
\item the quadric $Q_\sigma$ is irreducible.
\end{enumerate}

\end{recalltheorem}
\begin{recalltheorem}[{cf.\ Thm.\ \ref{thm: smooth}}]
Let $(X, \sigma)$ be a complex symplectic manifold, and assume that $X$ satisfies the $\del \delbar$-lemma. Let $Q_\sigma$ be the quadric in $\PP(H^2_{dR}(X, \C))$ defined by \eqref{eq: bbf}. Then the following are equivalent:
\begin{enumerate}
\item $q_\sigma$ is non-degenerate;
\item $h^{2, 0}(X) = 1$, i.e., $H^{2, 0}_{\delbar}(X) = \C \cdot [\sigma]$;
\item the quadric $Q_\sigma$ is smooth.
\end{enumerate}
\end{recalltheorem}

We collect through Section \ref{sect: bbf form} all the results we need for the proofs of the two main theorems and which are interesting also on their own; in particular Corollary \ref{cor: symplectic lefschetz} can be viewed as a sort of Lefschetz Theorem for complex symplectic manifolds. In Section \ref{sect: cones} we briefly study some property of the set
\[\{ [\sigma] \in H^{2, 0}_{\delbar}(X) | \sigma \text{ is a complex symplectic form} \} \subseteq H^{2, 0}_{\delbar}(X).\]
Finally, Section \ref{sect: examples} contains two examples: with the first we want to show the importance of the assumption that the $\del \delbar$-lemma holds for the manifolds we are dealing with, and with the second we want to show explicitly the results of the paper.

{\bf Acknowledgement:} the authors would like to thank Daniele Angella for useful comments. They want also thank the referee for his suitable suggestions.


\section{Preliminaries}\label{sect: preliminaries}

In this Section we want to recall the basic definitions and properties we will use in the sequel.

Let $X$ be a compact complex manifold and denote by $\cA^{p, q}(X)$ the space of smooth $(p, q)$-forms on $X$. According to Deligne, Griffiths, Morgan and Sullivan \cite{dgms}, $X$ is said to satisfy the $\del \delbar$-lemma if
\[\ker \del \cap \ker \delbar \cap \im d = \im \del \delbar.\]

The manifolds $X$ satisfying the $\del \delbar$-lemma (e.g., the K\"ahler or Moishezon manifolds as observed in \cite[Cor.\ 5.23]{dgms}) have interesting cohomological properties, which we will briefly sketch. Denote by
\[H^{p, q}_{BC}(X) = \frac{\ker (d: \cA^{p, q}(X) \longrightarrow \cA^{p + 1, q}(X) \oplus \cA^{p, q + 1}(X))}{\im (\del \delbar: \cA^{p - 1, q - 1}(X) \longrightarrow \cA^{p, q}(X))}\]
the \emph{Bott--Chern cohomology} of $X$, and by
\[H^{p, q}_{\delbar}(X) = \frac{\ker (\delbar: \cA^{p, q}(X) \longrightarrow \cA^{p, q + 1}(X))}{\im (\delbar: \cA^{p, q - 1}(X) \longrightarrow \cA^{p, q}(X))}\]
the \emph{Dolbeault cohomology} of $X$.

Then we have a natural homomorphism
\begin{equation}\label{eq: del-delbar iso}
\begin{array}{ccc}
H^{p, q}_{BC}(X) & \longrightarrow & H^{p, q}_{\delbar}(X)\\
{[\alpha]_{BC}} & \longmapsto & [\alpha]_{\delbar},
\end{array}
\end{equation}
which is an isomorphism if and only if the $\del \delbar$-lemma holds for $X$ (cf.\ \cite[Remark 5.16]{dgms}).

Two consequences of this fact (cf.\ \cite{dgms}) are that:
\begin{enumerate}
\item we have a decomposition
\[H^k(X, \C) = \bigoplus_{p + q = k} H^{p, q}_{\delbar}(X);\]
\item complex conjugation gives an isomorphism
\[\begin{array}{ccc}
H^{p, q}_{\delbar}(X) & \longrightarrow & H^{q, p}_{\delbar}(X)\\
{[\alpha]_{\delbar}} & \longmapsto & \overline{[\alpha]_{\delbar}} = [\bar{\alpha}]_{\delbar}.
\end{array}\]
\end{enumerate}

Let $(X, \sigma)$ be a \emph{complex symplectic manifold}, namely $X$ is an even dimensional (connected) compact complex manifold, and $\sigma$ is a $d$-closed $(2, 0)$-form which is non-degenerate at any point (cf.\ \cite[Def.\ 3.3]{boalch}). Observe that $\sigma$ is automatically holomorphic and that as direct consequences $X$ has trivial canonical bundle and $\dim H^{2, 0}_{\delbar}(X) \geq 1$.

To fix the notation, $(X, \sigma)$ will denote a compact complex symplectic manifold, $\dim_\C X = 2n$, and we will always assume that $\sigma$ is normalized, that is
\begin{equation}\label{eq: normalization}
\int_X (\sigma \bar{\sigma})^n = 1.
\end{equation}

\begin{rem}
Observe that our normalization assumption \eqref{eq: normalization} is not restrictive. In fact, as $\sigma^n$ is a nowhere vanishing canonical section, we have that $\dim H^{2n, 0}_{\delbar}(X) = 1$ and that $[\sigma^n]_{\delbar}$ is a generator for this space. By Serre duality $H^{0, 2n}_{\delbar}(X)$ is $1$-dimensional as well, and it is generated by $[\bar{\sigma}^n]_{\delbar}$ since
\begin{enumerate}
\item $\bar{\sigma}$ is a $d$-closed $(0, 2)$-form, and so it is also $\delbar$-closed;
\item $\bar{\sigma}^n$ is nowhere vanishing as $\sigma^n$ has the same property.
\end{enumerate}
As a consequence
\[0 \neq [\sigma^n]_{\delbar} \cup [\bar{\sigma}^n]_{\delbar} = \int_X (\sigma \bar{\sigma})^n\]
and so we can always assume this integral to be $1$.
\end{rem}

The \emph{Beauville--Bogomolov--Fujiki quadratic form} $q_\sigma$ on $H^2(X, \C)$ is defined, as in the hyperk\"ahler case (cf.\ \cite[$\S$8]{bea}), as follows: for each $[\alpha] \in H^2(X, \C)$ set
\begin{equation}\label{eq: bbf}
q_\sigma([\alpha]) = \frac{n}{2} \int_X (\sigma \bar{\sigma})^{n - 1} \alpha^2 + (1 - n) \left( \int_X \sigma^{n - 1} \bar{\sigma}^n \alpha \right) \left( \int_X \sigma^n \bar{\sigma}^{n - 1} \alpha \right).
\end{equation}

We will denote by $\langle \blank, \blank \rangle_\sigma$ the symmetric bilinear form polar to $q_\sigma$, namely
\[\langle [\alpha], [\beta] \rangle_\sigma = \frac{1}{2}(q_\sigma([\alpha] + [\beta]) - q_\sigma([\alpha]) - q_\sigma([\beta])).\]

Since $q_\sigma$ is homogeneous of degree $2$, we can consider the quadric $Q_\sigma$ defined by the equation $q_{\sigma} = 0$ in $\PP(H^2(X, \C))$. In the next Section we will mainly be concerned with the study of this quadric.

\section{The Beauville--Bogomolov--Fujiki quadric and the \texorpdfstring{$\del \delbar$-lemma}{del-delbar-lemma}}\label{sect: bbf form}

The purpose of this Section is to show the following results.

\begin{theorem}\label{thm: irreducible}
Let $(X, \sigma)$ be a complex symplectic manifold, and assume that $X$ satisfies the $\del \delbar$-lemma. Then the following are equivalent:
\begin{enumerate}
\item $h^{1, 1}(X) > 0$;
\item the quadric $Q_\sigma$ is irreducible.
\end{enumerate}
\end{theorem}

\begin{theorem}\label{thm: smooth}
Let $(X, \sigma)$ be a complex symplectic manifold, and assume that $X$ satisfies the $\del \delbar$-lemma. Then the following are equivalent:
\begin{enumerate}
\item $q_\sigma$ is non-degenerate;
\item $h^{2, 0}(X) = 1$, i.e., $H^{2, 0}_{\delbar}(X) = \C \cdot [\sigma]$;
\item the quadric $Q_\sigma$ is smooth.
\end{enumerate}
\end{theorem}

To be as clear as possible, we prove here the technical results needed for the Proofs of Theorem \ref{thm: irreducible} and Theorem \ref{thm: smooth}, which we postpone at the end of the Section.

As a matter of notations, $\Omega^1_X$ will be used to denote the holomorphic cotangent bundle of $X$, while $\Omega^p_X = \bigwedge^p \Omega^1_X$ and $K_X = \Omega^n_X$ will denote the bundles of holomorphic $p$-forms on $X$ and the canonical bundle of $X$ respectively. Finally, $L_{\alpha}$ will denote the operator given by wedging with $\alpha$.

\begin{lemma}\label{lemma: iso vbund}
Let $(X, \sigma)$ be a complex symplectic manifold of dimension $2n$. Then the map
\[L_\sigma^{n - 1}: \Omega^1_X \longrightarrow \Omega^{2n - 1}_X\]
given by the wedge product with $\sigma^{n - 1}$ is an isomorphism of vector bundles.
\end{lemma}
\begin{proof}
The Lemma follows from the description of this map in \cite{huy} and \cite{fujiki}, which we recall here. Since $\sigma$ is non-degenerate, it induces isomorphisms $T_X \simeq \Omega^1_X$ and $K_X \simeq \cO_X$. The first is given by
\[\begin{array}{ccc}
T_X & \longmapsto & \Omega^1_X\\
v & \longmapsto & \sigma(v, \blank),
\end{array}\]
while the second is induced by $\sigma^n$, which is holomorphic and everywhere non-vanishing. Since we have also a non-degenerate pairing $\Omega^1_X \otimes \Omega^{2n - 1}_X \longrightarrow K_X$ induced by the cup product, the fact that $K_X \simeq \cO_X$ exhibits also $\Omega^{2n - 1}_X$ as $(\Omega^1_X)^*$. So we have the sequence of isomorphisms
\[\Omega^1_X \simeq T_X \simeq (\Omega^1_X)^* \simeq \Omega^{2n - 1}_X,\]
and the map $L_\sigma^{n - 1}$ is nothing but this composition.
\end{proof}

As a consequence of Lemma \ref{lemma: iso vbund} the following Corollary holds.

\begin{corollary}\label{cor: symplectic lefschetz}
We have for the Dolbeault cohomology the following isomorphisms:
\[\xymatrix{L_\sigma^{n - 1}: & H^q(X, \Omega^1_X) \ar[r] \ar@{=}[d] & H^q(X, \Omega^{2n - 1}_X) \ar@{=}[d]\\
 & H^{1, q}_{\delbar}(X) & H^{2n - 1, q}_{\delbar}(X),}\]
where $L_\sigma^{n -1}$ is induced by the wedge product with $\sigma^{n - 1}$.
\end{corollary}

We can see this result as an instance of a sort of ``symplectic Lefschetz theorem''.


\begin{lemma}\label{lemma: iso (1, 1)}
Let $(X, \sigma)$ be a complex symplectic manifold, and assume that $X$ satisfies the $\del \delbar$-lemma. Then
\[\begin{array}{rccc}
L_{\sigma \bar{\sigma}}^{n - 1}: & H^{1, 1}_{\delbar}(X) & \longrightarrow & H^{2n - 1, 2n - 1}_{\delbar}(X)\\
 & [\alpha]_{\delbar} & \longmapsto & [(\sigma \bar{\sigma})^{n - 1} \alpha]_{\delbar}
\end{array}\]
is an isomorphism.
\end{lemma}
\begin{proof}
By Corollary \ref{cor: symplectic lefschetz}, $L_\sigma^{n - 1}: H^{1, q}_{\delbar}(X) \longrightarrow H^{2n - 1, q}(X)$ is an isomorphism. Since $X$ satisfies the $\del \delbar$-lemma, and so complex conjugation is an isomorphism, $L_{\bar{\sigma}}^{n - 1}: H^{q, 1}_{\delbar}(X) \longrightarrow H^{q, 2n - 1}(X)$ is an isomorphism as well, being $L_{\bar{\sigma}}^{n - 1}([\alpha]) = \overline{L_{\sigma}^{n - 1}([\bar{\alpha}])}$. In particular, for $q = 1$ we get our result since we have the commutative square
\[\xymatrix{H^{1, 1}_{\delbar}(X) \ar[r]^(0.45){L_\sigma^{n - 1}} \ar[d]^{L_{\bar{\sigma}}^{n - 1}} & H^{2n - 1, 1}_{\delbar}(X) \ar[d]^{L_{\bar{\sigma}}^{n - 1}}\\
H^{1, 2n - 1}_{\delbar}(X) \ar[r]^(0.45){L_\sigma^{n - 1}} & H^{2n - 1, 2n - 1}_{\delbar}(X)}\]
where all the maps are isomorphisms.
\end{proof}

\begin{rem}
Lemma \ref{lemma: iso (1, 1)} is false if $X$ does not satisfy the $\del \delbar$-lemma, as we will show in Example \ref{ex: no iso del delbar}.
\end{rem}

\begin{lemma}\label{lemma: non-degenerate}
Let $(X, \sigma)$ be a complex symplectic manifold, and assume that $X$ satisfies the $\del \delbar$-lemma. Then the restriction of $\langle \blank, \blank \rangle_\sigma$ to $H^{1, 1}_{\delbar}(X)$ is non-degenerate.
\end{lemma}
\begin{proof}
The symmetric bilinear form $\langle \blank, \blank \rangle_\sigma$ on $H^2(X, \C)$ restricts to the form on $H^{1, 1}_{\delbar}(X)$ given by
\[\langle [\alpha], [\beta] \rangle_\sigma = \frac{n}{2} \int_X (\sigma \bar{\sigma})^{n - 1} \alpha \beta.\]
Assume now that $[\alpha] \in H^{1, 1}_{\delbar}(X)$ is a class such that $\langle [\alpha], [\beta] \rangle_\sigma = 0$ for all $[\beta] \in H^{1, 1}_{\delbar}(X)$. This means that
\[\int_X L_{\sigma \bar{\sigma}}^{n - 1}([\alpha]) [\beta] = 0,\]
i.e., that $L_{\sigma \bar{\sigma}}^{n - 1}([\alpha])$ is a $(2n - 1, 2n - 1)$-class which annihilates all the space $H^{1, 1}_{\delbar}(X)$. By Serre duality, we have that then $L_{\sigma \bar{\sigma}}^{n - 1}([\alpha]) = 0$, and since we proved that $L_{\sigma \bar{\sigma}}^{n - 1}$ is an isomorphism, we finally get that $[\alpha] = 0$.
\end{proof}

\begin{proposition}\label{prop: (1, 1) perp}
Let $(X, \sigma)$ be a complex symplectic manifold, and assume that $X$ satisfies the $\del \delbar$-lemma. Then
\[H^{2, 0}_{\delbar}(X) \oplus H^{0, 2}_{\delbar}(X) = H^{1, 1}_{\delbar}(X)^\perp.\]
\end{proposition}
\begin{proof}
We will prove that $H^{2, 0}_{\delbar}(X) \oplus H^{0, 2}_{\delbar}(X) \subseteq H^{1, 1}_{\delbar}(X)^\perp$ and that the two spaces have the same dimension.
\begin{itemize}
\item[$(\subseteq)$] It is an easy computation with the polar form of $q_\sigma$. Let $[\alpha_{2, 0} + \alpha_{0, 2}] \in H^{2, 0}_{\delbar}(X) \oplus H^{0, 2}_{\delbar}(X)$, we need to prove that for any $[\alpha_{1, 1}] \in H^{1, 1}_{\delbar}(X)$ we have
\[\langle [\alpha_{2, 0} + \alpha_{0, 2}], [\alpha_{1, 1}] \rangle_\sigma = \langle [\alpha_{2, 0}], [\alpha_{1, 1}] \rangle_\sigma + \langle [\alpha_{0, 2}], [\alpha_{1, 1}] \rangle_\sigma = 0.\]
It is only a matter of type that
\[q_\sigma([\alpha_{2, 0}]) = q_\sigma([\alpha_{0, 2}]) = 0,\]
from which we deduce the equality
\[2 \langle [\alpha_{2, 0}], [\alpha_{1, 1}] \rangle_\sigma = q_\sigma([\alpha_{2, 0}] + [\alpha_{1, 1}]) - q_\sigma([\alpha_{2, 0}]) - q_\sigma([\alpha_{1, 1}]) = q_\sigma([\alpha_{1, 1}]) - q_\sigma([\alpha_{1, 1}]) = 0,\]
and similarly $\langle [\alpha_{0, 2}], [\alpha_{1, 1}] \rangle_\sigma = 0$. So we are done.
\item[$(=)$] We denote $b_i(X) = \dim H^i(X, \C)$ and $h^{p, q}(X) = \dim H^{p, q}_{\delbar}(X)$ as usual. By Lemma \ref{lemma: non-degenerate}, $\langle \blank, \blank \rangle_\sigma$ is non degenerate on $H^{1, 1}_{\delbar}(X)$, which is equivalent to $H^{1, 1}_{\delbar}(X) \cap H^{1, 1}_{\delbar}(X)^\perp = \{ 0 \}$. In turn this implies that $\dim H^{1, 1}_{\delbar}(X)^\perp = b_2(X) - h^{1, 1}(X)$, and finally that
\[\dim H^{1, 1}_{\delbar}(X)^\perp = b_2(X) - h^{1, 1}(X) = h^{2, 0}(X) + h^{0, 2}(X).\]
\end{itemize}
\end{proof}

Now we study the restriction of the quadratic form on $H^{2, 0}_{\delbar}(X) \oplus H^{0, 2}_{\delbar}(X)$.

\begin{lemma}\label{lemma: orthogonal condition}
Let $(X, \sigma)$ be a complex symplectic manifold satisfying the $\del \delbar$-lemma, and let $V = \Span\{ [\sigma], [\bar{\sigma}] \}$. Let $[\tau_{2, 0}] + [\tau_{0, 2}] \in H^{2, 0}_{\delbar}(X) \oplus H^{0, 2}_{\delbar}(X)$. Then $[\tau_{2, 0}] + [\tau_{0, 2}] \in V^\perp$ if and only if
\[\int_X \sigma^n \bar{\sigma}^{n - 1} \tau_{0, 2} = \int_X \sigma^{n - 1} \bar{\sigma}^n \tau_{2, 0} = 0.\]
\end{lemma}
\begin{proof}
We observe that $2 \langle [\sigma], [\tau_{2, 0}] + [\tau_{0, 2}] \rangle_\sigma = 0$ if and only if $q_\sigma([\sigma] + [\tau_{2, 0}] + [\tau_{0, 2}]) - q_\sigma([\tau_{2, 0}] + [\tau_{0, 2}]) = 0$; a straightforward computation shows that this last is equivalent to say that
\[\int_X \sigma^n \bar{\sigma}^{n - 1} \tau_{0, 2} = 0.\]
In the same way, one sees that $2 \langle [\bar{\sigma}], [\tau_{2, 0}] + [\tau_{0, 2}] \rangle_\sigma = 0$ if and only if $\int_X \sigma^{n - 1} \bar{\sigma}^n \tau_{2, 0} = 0$, and the result then follows.
\end{proof}

Observe that we have two natural linear functionals on $H^{2, 0}_{\delbar}(X)$ and $H^{0, 2}_{\delbar}(X)$ respectively:
\[\begin{array}{ccc}
H^{2, 0}_{\delbar}(X) & \longrightarrow & \C\\
{[\tau_{2, 0}]} & \longmapsto & \int_X \sigma^{n - 1} \bar{\sigma}^n \tau_{2, 0}
\end{array}, \qquad \begin{array}{ccc}
H^{0, 2}_{\delbar}(X) & \longrightarrow & \C\\
{[\tau_{0, 2}]} & \longmapsto & \int_X \sigma^n \bar{\sigma}^{n - 1} \tau_{0, 2}
\end{array}\]
These functionals are surjective since their value on $[\sigma]$, resp.\ $[\bar{\sigma}]$, is $1$, and so they have $(d - 1)$-dimensional kernel, where $d = h^{2, 0}(X) = h^{0, 2}(X)$. In particular, after fixing a basis $\{ [\sigma_2], \ldots, [\sigma_d] \}$ for the kernel of the first, we have
\[H^{2, 0}_{\delbar}(X) \oplus H^{0, 2}_{\delbar}(X) = \underbrace{\Span \{ \sigma, \bar{\sigma} \}}_{V} \oplus \underbrace{\Span \{ \sigma_2, \ldots, \sigma_d, \bar{\sigma}_2, \ldots, \bar{\sigma}_d \}}_{W},\]
where $W$ is the orthogonal complement of $V$ in $H^{2, 0}_{\delbar}(X) \oplus H^{0, 2}_{\delbar}(X)$ by Lemma \ref{lemma: orthogonal condition}.

We now want to show that $W$ is the subspace where the form $q_\sigma$ degenerates.

\begin{proposition}\label{prop: decomposition}
Let $(X, \sigma)$ be a complex symplectic manigold, and assume that $X$ satisfies the $\del \delbar$-lemma. Then the second complex cohomology group of $X$ has an orthogonal decomposition for $q_\sigma$
\[H^2(X, \C) = \Span\{ [\sigma], [\bar{\sigma}] \} \oplus W \oplus H^{1, 1}_{\delbar}(X),\]
where $W \subseteq H^{2, 0}_{\delbar}(X) \oplus H^{0, 2}_{\delbar}(X)$ is the kernel of $\langle \blank, \blank \rangle_\sigma$.
\end{proposition}
\begin{proof}
We know by Lemma \ref{lemma: non-degenerate} and Proposition \ref{prop: (1, 1) perp} that $\langle \blank, \blank \rangle_\sigma$ is non-degenerate on $H^{1, 1}_{\delbar}(X)$ and that $H^{1, 1}_{\delbar}(X)^\perp = H^{2, 0}_{\delbar}(X) \oplus H^{0, 2}_{\delbar}(X)$. Hence we can restrict our attention to $H^{2, 0}_{\delbar}(X) \oplus H^{0, 2}_{\delbar}(X)$, where the Beauville--Bogomolov--Fujiki form has expression
\[q_\sigma([\tau_{2, 0}] + [\tau_{0, 2}]) = n \int_X (\sigma \bar{\sigma})^{n - 1} \tau_{2, 0} \tau_{0, 2} + (1 - n) \left( \int_X \sigma^n \bar{\sigma}^{n - 1} \tau_{0, 2} \right) \left( \int_X \sigma^{n - 1} \bar{\sigma}^n \tau_{2, 0} \right).\]
Assume now that $[\tau_{2, 0}] + [\tau_{0, 2}] \in H^{2, 0}_{\delbar}(X) \oplus H^{0, 2}_{\delbar}(X)$ is orthogonal to $\Span\{ [\sigma], [\bar{\sigma}] \}$: by Lemma \ref{lemma: orthogonal condition} this means that
\[\left\{ \begin{array}{l}
\int_X \sigma^n \bar{\sigma}^{n - 1} \tau_{0, 2} = 0\\
\int_X \sigma^{n - 1} \bar{\sigma}^n \tau_{2, 0} = 0.
\end{array} \right.\]
Since $H^{2n, 0}_{\delbar}(X) = \C \cdot [\sigma^n]$, $H^{0, 2n}_{\delbar}(X) = \C \cdot [\bar{\sigma}^n]$ and the natural pairing is non-degenerate, we obtain that $[\bar{\sigma}^{n - 1} \tau_{0, 2}] = [\sigma^{n - 1} \tau_{2, 0}] = 0$, and consequently
\[q_\sigma([\tau_{2, 0}] + [\tau_{0, 2}]) = 0.\]
This computation shows that the bilinear form $\langle \blank, \blank \rangle_\sigma$ vanishes on $W$, and since it is non-degenerate on $\Span\{ [\sigma], [\bar{\sigma}] \}$ it follows that $W = \ker \langle \blank, \blank \rangle_\sigma$.
\end{proof}

We are now ready to prove the two Main Theorems.

\begin{proof}[Proof of Theorem \ref{thm: irreducible}]
The quadric $Q_\sigma$ is irreducible if and only if the Gram matrix of $\langle \blank, \blank, \rangle_\sigma$ has rank at least $3$. So, it suffices to show that there exist three orthogonal classes $[\alpha]_i \in H^2(X, \C)$ such that $q_\sigma([\alpha]_i) \neq 0$ if and only if $h^{1, 1}(X) > 0$. It is an easy computation that $[\sigma + \bar{\sigma}]$ and $[\sigma - \bar{\sigma}]$ are orthogonal and satisfy $q_\sigma([\sigma \pm \bar{\sigma}]) = \pm 1$. It follows then from Proposition \ref{prop: decomposition} and Lemma \ref{lemma: non-degenerate} that a third class with this property exists if and only if $h^{1, 1}(X) > 0$.
\end{proof}

\begin{proof}[Proof of Theorem \ref{thm: smooth}]
It is well known that a quadric is smooth if and only if it is defined by a non-degenerate quadratic form. From Proposition \ref{prop: decomposition} we have that $q_\sigma$ is non-degenerate if and only if (with the notations of Proposition \ref{prop: decomposition}) $W = \{ 0 \}$, i.e., if and only if $h^{2, 0}(X) = 1$.
\end{proof}

\begin{rem}
Assume that $(X, \sigma)$ is an \emph{irreducible holomorphic symplectic} manifold, i.e., that $X$ is a (compact) simply connected K\"ahler manifold such that $H^{2, 0}_{\delbar}(X) = \C \cdot [\sigma]$. Then it was shown in \cite[Th\'eor\`eme 5]{bea} that $Q_\sigma$ is smooth and irreducible. In particular, to show the irreducibility the proof uses the fact that the K\"ahler class $[\omega]$ of $X$ is orthogonal to $H^{2, 0}_{\delbar}(X) \oplus H^{0, 2}_{\delbar}(X)$ and satisfies
\[q_\sigma([\omega]) = \frac{n}{2} \int_X (\sigma \bar{\sigma})^{n - 1} \omega^2 > 0\]
by the Hodge--Riemann bilinear relations.
\end{rem}

\begin{rem}
Let $(X, \sigma)$ be a complex symplectic manifold. As $q_\sigma([\sigma]) = 0$, the point defined by $[\sigma]$ in $\PP(H^2_{dR}(X, \C))$ lies on $Q_\sigma$, and it follows from Proposition \ref{prop: decomposition} that $[\sigma]$ is always a smooth point of $Q_\sigma$.
\end{rem}

\section{Complex symplectic cones}\label{sect: cones}

In this section we want to study the locus
\[Sp(X) = \{ [\sigma] \in H^{2, 0}_{\delbar}(X) | \sigma \text{ is a complex symplectic form} \} \subseteq H^{2, 0}_{\delbar}(X).\]

Observe first of all that since there are no $\delbar$-boundaries in $\cA^{2, 0}(X)$, then we have only one element in each Dolbeault class, so
\[Sp(X) = \{ \sigma \in \cA^{2, 0}(X) | \sigma \text{ is a complex symplectic form} \} \subseteq \cZ^{2, 0}_{\delbar}(X) = \{ \varphi \in \cA^{2, 0}(X) | \delbar \varphi = 0 \}.\]

It is then easy to see from the definitions that $Sp(X)$ is a cone in $\cZ^{2, 0}_{\delbar}(X)$: if $\sigma \in Sp(X)$, then $\lambda \sigma \in Sp(X)$ for any $\lambda \in \C^*$.

\begin{proposition}
Let $X$ be a compact complex manifold satisfying the $\del \delbar$-lemma. Then the set $Sp(X)$ is open in $H^{2, 0}_{\delbar}(X)$.
\end{proposition}
\begin{proof}
Observe that if $\tau$ is a $\delbar$-closed $(2, 0)$-form, then $d \tau = 0$. In fact $d\tau = \del\tau$ is of type $(3, 0)$, and as $d\tau \in \ker \del \cap \ker \delbar \cap \im d = \im \del \delbar$ we deduce that $d\tau = 0$ because there are on $X$ no $(2, -1)$-forms. Let then $\sigma \in Sp(X)$ be fixed, and let $\tau \in \cZ^{2, 0}_{\delbar}(X)$ be any $(2, 0)$-form. Consider then the form $\sigma + \varepsilon \tau$, with $\varepsilon > 0$: this is
\begin{enumerate}
\item $\delbar$-closed, since both $\sigma$ and $\tau$ are;
\item $d$-closed, since both $\sigma$ and $\tau$ are;
\item non-degenerate for some suitable $\varepsilon > 0$, since $\sigma$ is non-degenerate and $X$ is compact.
\end{enumerate}
So we can perturb $\sigma \in Sp(X)$ in any direction remaining in $Sp(X)$, which implies that $Sp(X)$ is open.
\end{proof}

\section{Examples}\label{sect: examples}

As we observed in Corollary \ref{cor: symplectic lefschetz}, the map
\[\begin{array}{rccc}
L_\sigma^{n - 1}: & H^{1, q}_{\delbar}(X) & \longrightarrow & H^{2n - 1, q}_{\delbar}(X)\\
 & [\varphi]_{\delbar} & \longmapsto & [\sigma^{n - 1} \varphi]_{\delbar}
\end{array}\]
is an isomorphism.

In Example \ref{ex: no iso del delbar} we will show that this is not true if we consider Bott--Chern cohomology instead of Dolbeault cohomology, and that if $X$ does not satisfy the $\del \delbar$-lemma, then its conjugate $L_{\bar{\sigma}}^{n - 1}: H^{q, 1}_{\delbar}(X) \longrightarrow H^{q, 2n - 1}_{\delbar}(X)$ may fail to be an isomorphism. Finally, in Example \ref{ex: deformations} we want to deal with an explicit example that clarifies our results, in particular Theorems \ref{thm: irreducible} and \ref{thm: smooth}.

\begin{example}\label{ex: no iso del delbar}
Let $\I(3) = \Gamma \backslash G$ be the Iwasawa threefold, i.e., the quotient of the complex nilpotent lie group
\[G = \left\{ \left( \begin{array}{ccc}
1 & z_1 & z_3\\
0 & 1 & z_2\\
0 & 0 & 1
\end{array} \right) \,\,\middle|\,\, z_1, z_2, z_3 \in C \right\}\]
by the lattice $\Gamma$ consisting of the matrices with entries in $\Z[\sqrt{-1}]$. We have then on $\I(3)$ the following global $(1, 0)$-forms, expressed in terms of the natural coordinates $(z_1, z_2, z_3)$:
\[\varphi_1 = dz_1, \qquad \varphi_2 = dz_2, \qquad \varphi_3 = dz_3 - z_1 dz_2,\]
and which satisfy the structure equations
\[d\varphi_1 = 0, \qquad d\varphi_2 = 0, \qquad d\varphi_3 = -\varphi_1 \wedge \varphi_2.\]
We consider then $X = \I(3) \times \T$, where $\T$ is a complex $1$-dimensional torus with coordinate $z_4$ giving us a fourth $(1, 0)$-form $\varphi_4 = dz_4$ satisfying the structure equation $d\varphi_4 = 0$. Observe that $X$ is a $4$-dimensional complex manifold which does not satisfy the $\del \delbar$-lemma as $\I(3)$ does not. In the sequel we will use the following notation:
\[\varphi_{ij} = \varphi_i \wedge \varphi_j, \qquad \varphi_{i\bar{j}} = \varphi_i \wedge \bar{\varphi}_j, \qquad \ldots.\]
With some computations we can see that the second cohomologies of $X$ have the following generators: the de Rham cohomology is
\begin{equation}\label{eq: h^2 de rham}
H^2_{dR}(X, \C) = \Span \left\{ \begin{array}{c}
{[\varphi_{13}]_{dR}}, [\varphi_{14}]_{dR}, [\varphi_{23}]_{dR}, [\varphi_{24}]_{dR},\\
{[\varphi_{1\bar{1}}]_{dR}}, [\varphi_{1\bar{2}}]_{dR}, [\varphi_{1\bar{4}}]_{dR},\\
{[\varphi_{2\bar{1}}]_{dR}}, [\varphi_{2\bar{2}}]_{dR}, [\varphi_{2\bar{4}}]_{dR},\\
{[\varphi_{4\bar{1}}]_{dR}}, [\varphi_{4\bar{2}}]_{dR}, [\varphi_{4\bar{4}}]_{dR},\\
{[\varphi_{\bar{1}\bar{3}}]_{dR}}, [\varphi_{\bar{1}\bar{4}}]_{dR}, [\varphi_{\bar{2}\bar{3}}]_{dR}, [\varphi_{\bar{2}\bar{4}}]_{dR}
\end{array} \right\};
\end{equation}
the Bott--Chern cohomology is
\[\begin{array}{l}
H^{2, 0}_{BC}(X) = \Span \left\{ \begin{array}{c}
{[\varphi_{12}]_{BC}}, [\varphi_{13}]_{BC}, [\varphi_{14}]_{BC}, [\varphi_{23}]_{BC}, [\varphi_{24}]_{BC}
\end{array} \right\},\\
H^{1, 1}_{BC}(X) = \Span \left\{ \begin{array}{c}
{[\varphi_{1\bar{1}}]_{BC}}, [\varphi_{1\bar{2}}]_{BC}, [\varphi_{1\bar{4}}]_{BC},\\
{[\varphi_{2\bar{1}}]_{BC}}, [\varphi_{2\bar{2}}]_{BC}, [\varphi_{2\bar{4}}]_{BC},\\
{[\varphi_{4\bar{1}}]_{BC}}, [\varphi_{4\bar{2}}]_{BC}, [\varphi_{4\bar{4}}]_{BC}
\end{array} \right\},\\
H^{0, 2}_{BC}(X) = \Span \left\{ \begin{array}{c}
{[\varphi_{\bar{1}\bar{2}}]_{BC}}, [\varphi_{\bar{1}\bar{3}}]_{BC}, [\varphi_{\bar{1}\bar{4}}]_{BC}, [\varphi_{\bar{2}\bar{3}}]_{BC}, [\varphi_{\bar{2}\bar{4}}]_{BC}
\end{array} \right\};
\end{array}\]
and finally the Dolbeault cohomology is
\[\begin{array}{l}
H^{2, 0}_{\delbar}(X) = \Span \left\{ \begin{array}{c}
{[\varphi_{12}]_{\delbar}}, [\varphi_{13}]_{\delbar}, [\varphi_{14}]_{\delbar}, [\varphi_{23}]_{\delbar}, [\varphi_{24}]_{\delbar}, [\varphi_{34}]_{\delbar}
\end{array} \right\},\\
H^{1, 1}_{\delbar}(X) = \Span \left\{ \begin{array}{c}
{[\varphi_{1\bar{1}}]_{\delbar}}, [\varphi_{1\bar{2}}]_{\delbar}, [\varphi_{1\bar{4}}]_{\delbar},\\
{[\varphi_{2\bar{1}}]_{\delbar}}, [\varphi_{2\bar{2}}]_{\delbar}, [\varphi_{2\bar{4}}]_{\delbar},\\
{[\varphi_{3\bar{1}}]_{\delbar}}, [\varphi_{3\bar{2}}]_{\delbar}, [\varphi_{3\bar{4}}]_{\delbar},\\
{[\varphi_{4\bar{1}}]_{\delbar}}, [\varphi_{4\bar{2}}]_{\delbar}, [\varphi_{4\bar{4}}]_{\delbar}
\end{array} \right\},\\
H^{0, 2}_{\delbar}(X) = \Span \left\{ \begin{array}{c}
{[\varphi_{\bar{1}\bar{3}}]_{\delbar}}, [\varphi_{\bar{1}\bar{4}}]_{\delbar}, [\varphi_{\bar{2}\bar{3}}]_{\delbar}, [\varphi_{\bar{2}\bar{4}}]_{\delbar}
\end{array} \right\}.
\end{array}\]
For later purposes, we list here the generators for other cohomology groups:
\[\begin{array}{l}
H^{3, 1}_{\delbar}(X) = \Span \left\{ \begin{array}{c}
{[\varphi_{123\bar{1}}]_{\delbar}}, [\varphi_{123\bar{2}}]_{\delbar}, [\varphi_{123\bar{4}}]_{\delbar},\\
{[\varphi_{124\bar{1}}]_{\delbar}}, [\varphi_{124\bar{2}}]_{\delbar}, [\varphi_{124\bar{4}}]_{\delbar},\\
{[\varphi_{134\bar{1}}]_{\delbar}}, [\varphi_{134\bar{2}}]_{\delbar}, [\varphi_{134\bar{4}}]_{\delbar},\\
{[\varphi_{234\bar{1}}]_{\delbar}}, [\varphi_{234\bar{2}}]_{\delbar}, [\varphi_{234\bar{4}}]_{\delbar}
\end{array} \right\},\\
H^{3, 1}_{BC}(X) = \Span \left\{ \begin{array}{c}
{[\varphi_{123\bar{1}}]_{BC}}, [\varphi_{123\bar{2}}]_{BC}, [\varphi_{123\bar{4}}]_{BC},\\
{[\varphi_{124\bar{1}}]_{BC}}, [\varphi_{124\bar{2}}]_{BC}, [\varphi_{124\bar{4}}]_{BC},\\
{[\varphi_{134\bar{1}}]_{BC}}, [\varphi_{134\bar{2}}]_{BC}, [\varphi_{134\bar{4}}]_{BC},\\
{[\varphi_{234\bar{1}}]_{BC}}, [\varphi_{234\bar{2}}]_{BC}, [\varphi_{234\bar{4}}]_{BC}
\end{array} \right\},\\
H^{1, 3}_{\delbar}(X) = \Span \left\{ \begin{array}{c}
{[\varphi_{1\bar{1}\bar{2}\bar{3}}]_{\delbar}}, [\varphi_{1\bar{1}\bar{3}\bar{4}}]_{\delbar}, [\varphi_{1\bar{2}\bar{3}\bar{4}}]_{\delbar},\\
{[\varphi_{2\bar{1}\bar{2}\bar{3}}]_{\delbar}}, [\varphi_{2\bar{1}\bar{3}\bar{4}}]_{\delbar}, [\varphi_{2\bar{2}\bar{3}\bar{4}}]_{\delbar},\\
{[\varphi_{3\bar{1}\bar{2}\bar{3}}]_{\delbar}}, [\varphi_{3\bar{1}\bar{3}\bar{4}}]_{\delbar}, [\varphi_{3\bar{2}\bar{3}\bar{4}}]_{\delbar},\\
{[\varphi_{4\bar{1}\bar{2}\bar{3}}]_{\delbar}}, [\varphi_{4\bar{1}\bar{3}\bar{4}}]_{\delbar}, [\varphi_{4\bar{2}\bar{3}\bar{4}}]_{\delbar}
\end{array} \right\}
\end{array}\]
(see also \cite{angella-kasuya} for general computations and \cite{C-T} for Dolbeault formality).

We now focus on the complex symplectic forms on $X$. Let $\sigma$ be any $d$-closed form of type $(2, 0)$ on $X$: according to \eqref{eq: h^2 de rham} it is cohomologus to
\[\alpha \varphi_{12} + \beta \varphi_{13} + \gamma \varphi_{14} + \delta \varphi_{23} + \varepsilon \varphi_{24},\]
and it is easy to see by taking its square that such a form is non-degenerate (hence a symplectic form) if and only if
\[\beta \varepsilon - \gamma \delta \neq 0.\]
We now fix such a symplectic form $\sigma$, and we also assume that it satisfies the normalization
\[\int_X (\sigma \bar{\sigma})^2 = 4 |\beta \varepsilon - \gamma \delta|^2 \int_X \varphi_{1234\bar{1}\bar{2}\bar{3}\bar{4}} = 1.\]
It is then only a matter of computation with the generators given above that
\begin{enumerate}
\item the Lefschetz operator $L_\sigma: H^{1, 1}_{\delbar} \longrightarrow H^{3, 1}_{\delbar}(X)$ is an isomorphism (as stated in Corollary \ref{cor: symplectic lefschetz}), while on the Bott--Chern cohomology it defines only an injective homomorphism $L_\sigma: H^{1, 1}_{BC} \longrightarrow H^{3, 1}_{BC}(X)$;
\item the Lefschetz operator induced by $\bar{\sigma}$ is not an isomorphism, since
\[\ker (L_{\bar{\sigma}}: H^{1, 1}_{\delbar} \longrightarrow H^{1, 3}_{\delbar}(X)) = \Span \left\{
\begin{array}{c}
{[}\bar{\beta} \varphi_{1\bar{1}} + \bar{\delta} \varphi_{1\bar{2}}]_{\delbar}, [\bar{\beta} \varphi_{2\bar{1}} + \bar{\delta} \varphi_{2\bar{2}}]_{\delbar},\\
{[}\bar{\beta} \varphi_{3\bar{1}} + \bar{\delta} \varphi_{3\bar{2}}]_{\delbar}, [\bar{\beta} \varphi_{4\bar{1}} + \bar{\delta} \varphi_{4\bar{2}}]_{\delbar}
\end{array} \right\},\]
and so the hypothesis that $X$ satisfies the $\del \delbar$-lemma is necessary in Lemma \ref{lemma: iso (1, 1)};
\item the images under the canonical maps $H^{p, q}_{BC}(X) \longrightarrow H^{p + q}_{dR}(X, \C)$ of $H^{2, 0}_{BC}(X) \oplus H^{0, 2}_{BC}(X)$ and $H^{1, 1}_{BC}(X)$ are orthogonal with respect to $\langle \blank, \blank \rangle_\sigma$, and the form $q_\sigma$ is \emph{degenerate on both} of them.
\end{enumerate}
\end{example}

\begin{example}\label{ex: deformations}
Let $t \in \C \smallsetminus \{ 0 \}$ and consider the manifold $X_t$ which is the product of a deformation of the holomorphic parallelizable Nakamura threefold and a complex $1$-dimensional torus. The deformations of the Nakamura threefold we are considering were analysed in \cite{angella-kasuya}, where it is shown that they satisfy the $\del \delbar$-lemma: in terms of the natural coordinates $(z_1, z_2, z_3)$ on the threefold and $z_4$ on the torus the  manifold $X_t$ is described by
\[\begin{array}{llcll}
(1, 0)-\text{forms:} & \varphi_1 = dz_1 - t d\bar{z}_1 & \quad & (0, 1)-\text{forms:} & \omega_1 = d\bar{z}_1 - \bar{t} dz_1\\
 & \varphi_2 = e^{-z_1} dz_2 & \quad & & \omega_2 = e^{-z_1} d\bar{z}_2\\
 & \varphi_3 = e^{z_1} dz_3 & \quad & & \omega_3 = e^{z_1} d\bar{z}_3\\
 & \varphi_4 = dz_4 & \quad & & \omega_4 = d\bar{z}_4,
\end{array}\]
with structure equations
\[\begin{array}{lcl}
d\varphi_1 = 0 & \quad & d\omega_1 = 0\\
d\varphi_2 = -\frac{1}{1 - |t|^2} \varphi_1 \wedge \varphi_2 + \frac{t}{1 - |t|^2} \varphi_2 \wedge \omega_1 & \quad & d\omega_2 = -\frac{1}{1 - |t|^2} \varphi_1 \wedge \omega_2 - \frac{t}{1 - |t|^2} \omega_1 \wedge \omega_2\\
d\varphi_3 = \frac{1}{1 - |t|^2} \varphi_1 \wedge \varphi_3 - \frac{t}{1 - |t|^2} \varphi_3 \wedge \omega_1 & \quad & d\omega_3 = \frac{1}{1 - |t|^2} \varphi_1 \wedge \omega_3 + \frac{t}{1 - |t|^2} \omega_1 \wedge \omega_3\\
d\varphi_4 = 0 & \quad & d\omega_4 = 0.
\end{array}\]
We observe that among all the $(2, 0)$-forms, the one which are $d$-closed are those of the form
\[\alpha \varphi_{14} + \beta \varphi_{23}, \qquad \alpha, \beta \in \C\]
and such forms are non-degenerate if and only if $\alpha, \beta \in \C^*$. This shows that $Sp(X_t) \subseteq H^{2, 0}_{\delbar}(X_t)$ is nothing but $(\C^*)^2 \subseteq \C^2$, which is clearly an open cone. It is then possible to compute explicitly the cohomology of our manifold, and in particular its second cohomology spaces:
\begin{equation}\label{eq: h^2}
\begin{array}{l}
H^{2, 0}_{\delbar}(X_t) = \Span \{ [\varphi_{14}], [\varphi_{23}] \},\\
H^{1, 1}_{\delbar}(X_t) = \Span \{ [\varphi_1 \wedge \omega_1], [\varphi_1 \wedge \omega_4], [\varphi_2 \wedge \omega_3], [\varphi_3 \wedge \omega_2], [\varphi_4 \wedge \omega_1], [\varphi_4 \wedge \omega_4] \},\\
H^{0, 2}_{\delbar}(X_t) = \Span \{ [\omega_{14}], [\omega_{23}] \}.
\end{array}
\end{equation}
As a consequence of the $\del \delbar$-lemma, each Dolbeault class has a representative which is $d$-closed, and in \eqref{eq: h^2} we used such representatives. We can also describe the action of complex conjugation:
\[\begin{array}{ccccc}
\overline{[\varphi_{14}]} = [\omega_{14}] & & \overline{[\varphi_1 \wedge \omega_1]} = -[\varphi_1 \wedge \omega_1] & & \overline{[\varphi_1 \wedge \omega_4]} = -[\varphi_4 \wedge \omega_1]\\
\overline{[\varphi_{23}]} = [\omega_{23}] & & \overline{[\varphi_2 \wedge \omega_3]} = -[\varphi_3 \wedge \omega_2] & & \overline{[\varphi_4 \wedge \omega_4]} = -[\varphi_4 \wedge \omega_4].
\end{array}\]
Fix now a complex symplectic form $\sigma = \alpha \varphi_{14} + \beta \varphi_{23}$, with $\alpha, \beta \in \C^*$, and let $\bar{\sigma} = \bar{\alpha} \omega_{14} + \bar{\beta} \omega_{23}$ be its complex conjugate. By our normalization assumption \eqref{eq: normalization}, we will assume that
\[\int_{X_t} (\sigma \bar{\sigma})^2 = \int_X 4 |\alpha|^2 |\beta|^2 \varphi_{1234} \wedge \omega_{1234} = 1.\]
Writing now
\[\begin{array}{rl}
[\alpha] = & a_{14} [\varphi_{14}] + a_{23} [\varphi_{23}] + \\
+ & b_{11} [\varphi_1 \wedge \omega_1] + b_{14} [\varphi_1 \wedge \omega_4] + b_{23} [\varphi_2 \wedge \omega_3] + b_{32} [\varphi_3 \wedge \omega_2] + b_{41} [\varphi_4 \wedge \omega_1] + b_{44} [\varphi_4 \wedge \omega_4] + \\
+ & c_{14} [\omega_{14}] + c_{23} [\omega_{23}]
\end{array}\]
for the generic class $[\alpha] \in H^2(X, \C)$, we can compute explicitly $q_\sigma([\alpha])$. Using on $H^2(X, \C)$ the ordered set of coordinates $(a_{14}, a_{23}, b_{11}, b_{14}, b_{23}, b_{32}, b_{41}, b_{44}, c_{14}, c_{23})$ we have just introduced, the Beauville--Bogomolov--Fujiki form on $X_t$ is described by the matrix
\[\left( \begin{array}{cc|cccccc|cc}
0 & 0 & 0 & 0 & 0 & 0 & 0 & 0 & \frac{1}{8 |\alpha|^2} & \frac{1}{8 \alpha \bar{\beta}}\\
0 & 0 & 0 & 0 & 0 & 0 & 0 & 0 & \frac{1}{8 \beta \bar{\alpha}} & \frac{1}{8 |\beta|^2}\\
\hline
0 & 0 & 0 & 0 & 0 & 0 & 0 & -\frac{1}{4 |\alpha|^2} & 0 & 0\\
0 & 0 & 0 & 0 & 0 & 0 & \frac{1}{4 |\alpha|^2} & 0 & 0 & 0\\
0 & 0 & 0 & 0 & 0 & \frac{1}{4 |\beta|^2} & 0 & 0 & 0 & 0\\
0 & 0 & 0 & 0 & \frac{1}{4 |\beta|^2} & 0 & 0 & 0 & 0 & 0\\
0 & 0 & 0 & \frac{1}{4 |\alpha|^2} & 0 & 0 & 0 & 0 & 0 & 0\\
0 & 0 & -\frac{1}{4 |\alpha|^2} & 0 & 0 & 0 & 0 & 0 & 0 & 0\\
\hline
\frac{1}{8 |\alpha|^2} & \frac{1}{8 \beta \bar{\alpha}} & 0 & 0 & 0 & 0 & 0 & 0 & 0 & 0\\
\frac{1}{8 \alpha \bar{\beta}} & \frac{1}{8 |\beta|^2} & 0 & 0 & 0 & 0 & 0 & 0 & 0 & 0\\
\end{array} \right),\]
from which it is possible to observe that
\begin{enumerate}
\item the form is degenerate, and its kernel is $\Span \{ \alpha [\varphi_{14}] - \beta [\varphi_{23}], \bar{\alpha} [\omega_{14}] - \bar{\beta} [\omega_{23}] \}$,
\item the signature of this form is $(p_+, p_-, p_0) = (4, 4, 2)$, and its restriction to $H^{1, 1}_{\delbar}(X_t)$ is non-degenerate of signature $(p_+, p_-) = (3, 3)$.
\end{enumerate}
According to Theorem \ref{thm: irreducible} and Theorem \ref{thm: smooth}, we see that the quadric defined in $\PP^9$ by this matrix is irreducible and singular.
\end{example}

\end{document}